\documentclass[12pt]{article}
\usepackage{amssymb}
\usepackage{amsfonts}
\usepackage{amsmath}
\usepackage[usenames]{color}
\usepackage{mathrsfs}
\usepackage{amsfonts}
\usepackage{amssymb,amsmath}
\usepackage{CJK}
\usepackage{cite}
\usepackage{cases}
\usepackage{amsthm}

\def\cl{\centerline}

\def\vs{\vspace*}

\def\L{\mathcal{L}}

\def\G{\mathcal{G}}
\def\Z{\mathbb{Z}}
\def\N{\mathbb{N}}

\def\C{\mathbb{C}}

\def\ni{\noindent}

\numberwithin{equation}{section}
\newtheorem{theo}{Theorem}[section]
\newtheorem{defi}[theo]{Definition}
\newtheorem{coro}[theo]{Corollary}
\newtheorem{lemm}[theo]{Lemma}
\newtheorem{exam}[theo]{Example}

\newtheorem{clai}{Claim}
\newtheorem{case}{Case}

\newtheorem{rema}[theo]{Remark}

\newtheorem{remark}[theo]{Remark}

\begin{document}
\begin{center}
\cl{\large\bf \vs{6pt} A family of new simple modules over the}
\cl{\large\bf Schr\"{o}dinger-Virasoro algebra\,{$^*\,$}}
\footnote {$^*\,$ Supported by the National Natural Science Foundation of China (No. 11371278, 11431010, 11501515) and the Zhejiang Provincial Natural Science Foundation of China (No.~LQ16A010011).
\\\indent\ \ $^\dag\,$ rebel1025@126.com
}
\cl{ Haibo Chen$^{1,\dag}$, Yanyong Hong$^{2}$, Yucai Su$^1$}

\cl{\small 1. Department of Mathematics, Tongji University, Shanghai
200092, China}
\cl{\small 2. Department of Science, Zhejiang Agriculture and Forestry University, Hangzhou 311300, China}
\end{center}

{\small
\parskip .005 truein
\baselineskip 3pt \lineskip 3pt

\noindent{{\bf Abstract:}
In this article, a large  class of simple modules over the  Schr\"{o}dinger-Virasoro algebra $\mathcal{G}$ are constructed, which
include highest weight modules and Whittaker modules.
These modules are determined by the simple modules over the finite-dimensional quotient algebras of some subalgebras. Moreover, we show that all simple modules of $\mathcal{G}$ with
locally finite actions of elements in a certain positive part belong to this class of simple modules. Similarly, a large  class of simple modules over the  $W$-algebra $W(2,2)$ are constructed.
\vs{5pt}

\ni{\bf Key words:}
 Schr\"{o}dinger-Virasoro
 algebra, $W$-algebra, Highest weight module, Whittaker module, Simple module.}

\ni{\it Mathematics Subject Classification (2010):} 17B10, 17B20, 17B65, 17B66, 17B68.}
\parskip .001 truein\baselineskip 6pt \lineskip 6pt
\section{Introduction}
Throughout this paper, we denote by $\C,\Z,\N$  and $\Z_+$ the sets of complex numbers, integers, nonnegative integers and positive integers, respectively.
 All vector spaces and Lie algebras are over $\C$.
 For a Lie algebra $\mathcal{L}$, we denote by $\mathcal{U}(\mathcal{L})$ the    universal enveloping algebra of $\mathcal{L}$.

The Schr\"{o}dinger-Virasoro algebra  is an extension of the Virasoro Lie algebra by
a nilpotent Lie algebra formed with a bosonic current of weight $\frac{3}{2}$
 and a bosonic current of
weight $1$. It was introduced
in the context of non-equilibrium statistical physics during the process
of investigating the free Schr\"{o}dinger   equations (see \cite{M1}).
From then on, the Schr\"{o}dinger-Virasoro algebra  attracted a lot of
attentions from researchers (see, e.g., \cite{TZ,RU,U,HLS,LS,ZT} ).
Now we recall the definition of the  {\em Schr\"{o}dinger-Virasoro algebra} $\mathcal{G}$, which is
 an infinite dimensional  Lie algebra
with the $\C$-basis  $\{M_m,Y_{m+\frac{1}{2}},L_m,C
\mid m\in \Z\}$ and the following   Lie brackets:
\begin{equation}\label{def1.1}
\aligned
&[L_m,L_n]= (n-m)L_{m+n}+\delta_{m+n,0}\frac{m^{3}-m}{12}C,\\&
[L_m,Y_{n+\frac{1}{2}}]= \Big(n+\frac{1-m}{2}\Big)Y_{m+n+\frac{1}{2}},\
 [Y_{m+\frac{1}{2}},Y_{n+\frac{1}{2}}]= (n-m)M_{m+n+1},\\&
 [L_m,M_n]=n M_{m+n},\ [M_m,M_n]=[M_m,Y_{n+\frac{1}{2}}]=[\G,C]=0, \quad \forall  m,n\in\Z.
\endaligned
\end{equation}
Note that  the center of $\G$ is spanned by $\{M_0,C\}$. In addition, the  Schr\"{o}dinger-Virasoro algebra is a special case for the
generalized Schr\"{o}dinger-Virasoro algebra (see \cite{TZ}).

 The highest weight modules and  Whittaker modules are  two classes of important modules,
especially for
infinite-dimensional Lie algebras such as Virasoro algebras, Heisenberg-Virasoro algebras,  affine Kac-Moody algebras and so on.
The construction of highest weight modules is one of the efficient ways to obtain simple weight modules  (see, e.g., \cite{CG,KV,KR,MZ,HJ,MZ1}),
while Whittaker modules become popular in recent years.
It is well-known that Whittaker modules for $sl(2)$ were first discovered  by Arnal and Pinczon in \cite{AP}.
At the same time, the
versions of Whittaker modules for finite dimensional complex semisimple Lie algebras were
introduced  by   Kostant (see \cite{K}).
Since then, Whittaker modules over various Lie algebras  draw a lot of
attentions from the mathematicians and physicists
(see, e.g., \cite{ME1,ME2,OW,LWZ,GL,ALZ}). Moreover, the Verma modules and Whittaker modules for the  Schr\"{o}dinger-Virasoro algebra are investigated in  \cite{TZ} and \cite{ZTL}, respectively.

The actions of elements in the positive part of the algebra are locally finite,
which is the same property of highest weight modules and Whittaker modules.
This makes us study such a class of modules in a uniform way.
Motivated by \cite{CG,MZ},  we construct  a large  family of new simple modules over  the Schr\"{o}dinger-Virasoro algebra.
The highest weight modules and Whittaker modules are included and other modules, which are not weight modules, are new. Moreover, a class of new simple modules of $W(2,2)$ are constructed similarly.

Let us now briefly describe how this  paper is organized.
 In Section 2, we recall some fundamental definitions about what we need in the following. In Section 3, a class of new $\mathcal{G}$-modules are constructed, which are induced from simple modules over the finite-dimensional
quotient algebras of  some subalgebras. This result recovers one of the main results of   Verma modules for a special case   in \cite{TZ} and the main results of  Whittaker modules in \cite{ZTL}.
In addition,
it is shown that any simple module with   locally finite actions of $M_i,(1-\delta_{i,0})Y_{i-\frac{1}{2}},L_i$ (or equivalently, locally nilpotent actions of $M_i,(1-\delta_{i,0})Y_{i-\frac{1}{2}},L_i$ as we shall prove) for
sufficiently large $i\in\N$ must be one of the modules constructed above.
In Section 4, some examples of simple $\mathcal{G}$-modules are presented. Finally, by the similar method, we describe new simple modules
over the $W$-algebra $W(2,2)$.

The main results of this paper are summarized in Theorems \ref{th1}, \ref{th2}, \ref{th51} and \ref{th52}.
\section{Preliminaries}
In this section, we shall construct a class of   induced   $\G$-modules $\mathrm{Ind}(V)$, where $V$ is a simple module. First, some definitions and  results  for later use are recalled.
\begin{defi}\rm
Let $V$ be a module for a Lie algebra $\L$ and $x\in\L$.

{\rm (1)} If for any $v\in V$ there exists $n\in\Z_+$ such that $x^nv=0$,  then we call that the action of  $x$  on $V$ is {\em locally nilpotent}.
Similarly, the action of $\L$   on $V$  is {\em locally nilpotent} if for any $v\in V$ there exists $n\in\Z_+$ such that $\L^nx=0$.

{\rm (2)} If for any $v\in V$ we have $\mathrm{dim}(\sum_{n\in\Z_+}\C x^nv)<+\infty$, then we call that the action of $x$ on $V$ is  {\em locally finite}.
Similarly, the action of $\L$  on $V$  is  {\em locally finite}
 if for any $v\in V$ we have $\mathrm{dim}(\sum_{n\in\Z_+} {\L}^nv)<+\infty$.
\end{defi}
It is easy to see that  the action of $x$ on $V$ is locally nilpotent implies that the action of $x$ on $V$ is locally finite.
If $\L$ is a finitely generated Lie algebra, then we have that the action of $\L$ on $V$ is locally nilpotent implies that the action of
 $\L$
 on $V$ is locally finite.

Denote by  $\mathbb{M}$  the set of all infinite vectors of the form $\underline{i}:=(\ldots, i_2, i_1)$ with entries in $\N$,
satisfying the condition that the number of nonzero entries is finite. Let $\underline{0}$ denote the element $(\ldots, 0, 0)\in\mathbb{M}$ and
for
$i\in\Z_+$ let $\epsilon_i$ denote the element $(\ldots,0,1,0,\ldots,0)\in\mathbb{M}$,
where $1$ is
in the $i$'th  position from right. For any $\underline{i}\in\mathbb{M}$, we write
$$\mathbf{w}(\underline{i})=\sum_{s\in\Z_+}s\cdot i_s,$$
which is a nonnegative integer. For any nonzero  $\underline{i}\in\mathbb{M}$, let $p$ and $q$ be the largest and smallest integers such that $i_p	 \neq0$ and $i_q	\neq0$ respectively,
and define  $\underline{i}^\prime=\underline{i}-\epsilon_p$
and  $\underline{i}^{\prime\prime}=\underline{i}-\epsilon_q$.

\begin{defi}\rm
{\rm (1)} Denote by $>$ the   {\em lexicographical total order}  on  $\mathbb{M}$, defined as follows: for any $\underline{i},\underline{j}\in\mathbb{M}$
$$\underline{j} >  \underline{i} \ \Leftrightarrow \ \mathrm{ there\ exists} \ r\in\Z_+ \ \mathrm{such \ that} \ (j_s=i_s,\ \forall s>r) \ \mathrm{and} \ j_r>i_r.$$
{\rm (2)} Denote by $\succ$ the   {\em reverse  lexicographical  total order}  on  $\mathbb{M}$,  defined as follows: for any $\underline{i},\underline{j}\in\mathbb{M}$
$$\underline{j} \succ \underline{i} \ \Leftrightarrow \ \mathrm{ there\ exists} \ r\in\Z_+ \ \mathrm{such \ that} \ (j_s=i_s,\ \forall 1\leq s<r) \ \mathrm{and} \ j_r>i_r.$$
Now we can induce a {\em principal total order} on $\mathbb{M}\times\mathbb{M}\times\mathbb{M}$, still denoted by $\succ$:
\begin{equation*}\aligned
(\underline{i},\underline{j},\underline{k}) \succ (\underline{l},\underline{m},\underline{n})\  \Leftrightarrow \ \   & (\underline{k},\mathbf{w}(\underline{k}))\succ(\underline{n},\mathbf{w}(\underline{n})) \quad \mathrm{or}\\&
\underline{k}=\underline{n}\
\mathrm{and}\ (\underline{j},\mathbf{w}(\underline{j})) \succ (\underline{m},\mathbf{w}(\underline{m}))
\quad \mathrm{or}\\&
 \underline{k}=\underline{n}, \ \underline{j}=\underline{m} \
\mathrm{and }\ \underline{i} > \underline{l},\quad \forall\underline{i},\underline{j},\underline{k},\underline{l},\underline{m},\underline{n}\in\mathbb{M}.
\endaligned
\end{equation*}
\end{defi}
For any $d_1,d_2\in\N$ with $d_1\geq2d_2-1$,   set
$$\G_{d_1,d_2}=\sum_{i\in\N}(\C M_{i-d_1}\oplus\C (1-\delta_{i,0})Y_{i-d_2-\frac{1}{2}}\oplus\C L_i)\oplus\C C.$$
Then, it is easy to see that $\G_{d_1,d_2}$ is a subalgebra of $\G$.

Letting  $V$ be a simple $\G_{d_1,d_2}$-module, then we have the induced   $\G$-module
$$\mathrm{Ind}(V)=\mathcal{U}(\G)\otimes_{\mathcal{U}(\G_{d_1,d_2})}V.$$

Since simple modules over one of subalgebras of $\G$ containing the central elements $M_0$ and $C$ are usually considered in the following,
we always assume that the actions of $M_0$ and $C$ are scalars $\nu_0$ and $c$ respectively.

Fix $d_1,d_2\in\N$  and let $V$ be a simple $\G_{d_1,d_2}$-module. For $\underline{i},\underline{j},\underline{k}\in\mathbb{M}$, we denote
$$M^{\underline{i}} Y^{\underline{j}} L^{\underline{k}}=\ldots M_{-d_1-2}^{i_2} M_{-d_1-1}^{i_1}\ldots Y_{-d_2-\frac{3}{2}}^{j_2} Y_{-d_2-\frac{1}{2}}^{j_1}\ldots L_{-2}^{k_2} L_{-1}^{k_1}\in \mathcal{U}(\mathcal{G}).$$
According to the $\mathrm{PBW}$ Theorem, every element of $\mathrm{Ind}(V)$ can be uniquely written in the
following form
\begin{equation}\label{def2.1}
\sum_{\underline{i},\underline{j},\underline{k}\in\mathbb{M}}M^{\underline{i}} Y^{\underline{j}} L^{\underline{k}} v_{\underline{i},\underline{j},\underline{k}},
\end{equation}
where all  $v_{\underline{i},\underline{j},\underline{k}}\in V$ and only finitely many of them are nonzero. For any $v\in\mathrm{Ind}(V)$ as in  \eqref{def2.1}, we denote by $\mathrm{supp}(v)$ the set of all $(\underline{i},\underline{j},\underline{k})\in \mathbb{M}\times \mathbb{M}\times\mathbb{M}$  such that $v_{\underline{i},\underline{j},\underline{k}}\neq0$.
 For a nonzero $v\in \mathrm{Ind}(V)$, we write $\mathrm{deg}(v)$  the maximal (with respect to the principal total order on $\mathbb{M}\times\mathbb{M}\times\mathbb{M}$) element in $\mathrm{supp}(v)$, called the {\em degree} of $v$. Note that here and later
 we make the
convention that  $\mathrm{deg}(v)$ only for  $v\neq0$.

\section{Characterization of simple modules}
The purpose of this section is to state two main results of this paper. We first prove that the   induced $\mathcal{G}$-module $\mathrm{Ind}(V)$ is simple  under certain   conditions which appeared in  Theorem \ref{th1}.  Then we shall show that under the   conditions in  Theorem \ref{th1}  any simple $\mathcal{G}$-module with locally finite actions of elements $M_i,(1-\delta_{j,0})Y_{j-\frac{1}{2}},L_k$ for sufficiently large $i,j,k\in\N$ is isomorphic to one of the induced  $\mathcal{G}$-modules $\mathrm{Ind}(V)$.

 Now we can summarize the key result in this section as follows.
\begin{theo}\label{th1}
Let $d_1,d_2\in\N$ and $V$ be a simple $\G_{d_1,d_2}$-module and  there exists  $t\in\N$  satisfying the following two conditions:

\noindent
{\rm (a)} the action of  $M_{t}$   on  $V$ is injective;

\noindent
{\rm (b)}   $M_iV=Y_{j-\frac{1}{2}}V=L_kV=0$ for  all $i>t,j>t+d_2$ and $k>t+d_1$.

\noindent
Then we have

\noindent
{\rm (1)} $\mathrm{Ind}(V)$ is a simple $\mathcal{G}$-module;

\noindent
{\rm (2)} the actions of  $M_i,Y_{j-\frac{1}{2}},L_k$ on $\mathrm{Ind}(V)$  for  all $i>t,j>t+d_2$ and $k>t+d_1$  are locally nilpotent.
\end{theo}

\begin{proof}
In order to prove part $(1)$ of Theorem \ref{th1}, we first introduce the following claim.
 \begin{clai}\label{claim3.1}
For any $v\in\ \mathrm{Ind}(V)\setminus V$, let $\mathrm{deg}(v)=(\underline{i},\underline{j},\underline{k})$,
$\hat{i}=\mathrm{max}\{s:i_s\neq0\}$ if $\underline{i}\neq\underline{0}$, ${\hat{j}}=\mathrm{min}\{s:j_s\neq0\}$ if $\underline{j}\neq\underline{0}$ and  $\hat{k}=\mathrm{min}\{s:k_s\neq0\}$ if $\underline{k}\neq\underline{0}$.  Then we obtain

\noindent
{\rm (1)}   if $\underline{k}\neq\underline{0}$,  then $\hat{k}>0$ and $\mathrm{deg}(M_{\hat{k}+t}v)=(\underline{i},\underline{j},\underline{k}^{\prime\prime})$;

\noindent
{\rm (2)}  if $\underline{k}=\underline{0},\underline{j}\neq\underline{0}$,  then ${\hat{j}}>0$ and $\mathrm{deg}(Y_{{\hat{j}}+t+d_2-\frac{1}{2}}v)=(\underline{i},\underline{j}^{\prime\prime},0)$;

\noindent
{\rm (3)} if $\underline{j}=\underline{k}=\underline{0},\underline{i}\neq\underline{0}$,  then $\hat{i}>0$ and $\mathrm{deg}(L_{\hat{i}+t+d_1}v)=(\underline{i}^{\prime},0,0)$.
\end{clai}
To prove this, we assume that $v$ is of the form in \eqref{def2.1}.

{\rm (1)} It is enough to show that we want to have  by comparing the degree. Now we consider those $v_{\underline{x},\underline{y},\underline{z}}$
with $$M_{\hat{k}+t}M^{\underline{x}}Y^{\underline{y}}L^{\underline{z}}v_{\underline{x},\underline{y},\underline{z}}\neq0.$$
Note that $M_{\hat{k}+t}v_{\underline{x},\underline{y},\underline{z}}=0$ for any  $(\underline{x},\underline{y},\underline{z})\in \mathrm{supp}(v)$.
One can easily check that
$$M_{\hat{k}+t}M^{\underline{x}}Y^{\underline{y}}L^{\underline{z}}v_{\underline{x},\underline{y},\underline{z}}=
M^{\underline{x}}Y^{\underline{y}}[M_{\hat{k}+t},L^{\underline{z}}]v_{\underline{x},\underline{y},\underline{z}}.$$
By (a), $M_{t}v_{\underline{x},\underline{y},\underline{z}}\neq0$.
If $\underline{z}=\underline{k}$, it is easy to see that
$$\mathrm{deg}(M_{\hat{k}+t}M^{\underline{x}}Y^{\underline{y}}L^{\underline{z}}v_{\underline{x},\underline{y},\underline{z}})=
(\underline{x},\underline{y},\underline{k}^{\prime\prime})\preceq (\underline{i},\underline{j},\underline{k}^{\prime\prime}),$$
where the equality holds if and only if $\underline{y}=\underline{j},\underline{x}=\underline{i}$.

Now we suppose  $(\underline{z},\mathbf{w}(\underline{z}))\prec(\underline{k},\mathbf{w}(\underline{k}))$ and denote  $$\mathrm{deg}(M_{\hat{k}+t}M^{\underline{x}}Y^{\underline{y}}L^{\underline{z}}v_{\underline{x},\underline{y},\underline{z}})
=(\underline{x}_1,\underline{y}_1,\underline{z}_1)\in \mathbb{M}\times \mathbb{M}\times \mathbb{M}.$$ If $\mathbf{w}(\underline{z})<\mathbf{w}(\underline{k})$,
then we get $\mathbf{w}(\underline{z}_1)\leq \mathbf{w}(\underline{z})-\hat{k}<\mathbf{w}(\underline{k})-\hat{k}=\mathbf{w}(\underline{k}^{\prime\prime})$,
which gives rise to $(\underline{x}_1,\underline{y}_1,\underline{z}_1)\prec(\underline{i},\underline{j},\underline{k}^{\prime\prime})$.

Then we suppose $\mathbf{w}(\underline{z})=\mathbf{w}(\underline{k})$ and ${\underline{z}}\prec{\underline{k}}$. Let $\hat{z}:=\mathrm{min}\{s:z_s\neq0\}>0$.
If $\hat{z}>\hat{k}$, it is easy to see that $\mathbf{w}(\underline{z}_1)<\mathbf{w}(\underline{z})-\hat{k}=\mathbf{w}(\underline{k}^{\prime\prime})$. If $\hat{z}=\hat{k}$,
we can similarly deduce $(\underline{x}_1,\underline{y}_1,\underline{z}_1)=(\underline{x},\underline{y},\underline{z}^{\prime\prime})$.
Since $\underline{z}^{\prime\prime}\prec \underline{k}^{\prime\prime}$, we have $\mathrm{deg}(M_{\hat{k}+t}M^{\underline{x}}Y^{\underline{y}}L^{\underline{z}}v_{\underline{x},\underline{y},\underline{z}})
=(\underline{x}_1,\underline{y}_1,\underline{z}_1)\prec(\underline{i},\underline{j},\underline{k}^{\prime\prime})$ in both cases.

Combining all the arguments above we conclude that
$\mathrm{deg}(M_{\hat{k}+t}v)=(\underline{i},\underline{j},\underline{k}^{\prime\prime})$, as desired.

{\rm (2)} Now we use   the similar method that appeared in  above.  We consider   $v_{\underline{x},\underline{y},\underline{0}}$
with $$Y_{{\hat{j}}+t+d_2-\frac{1}{2}}M^{\underline{x}}Y^{\underline{y}}v_{\underline{x},\underline{y},\underline{0}}\neq0.$$
Since $Y_{{\hat{j}}+t+d_2-\frac{1}{2}}v_{\underline{x},\underline{y},\underline{0}}=0$ for any  $(\underline{x},\underline{y},\underline{0})\in \mathrm{supp}(v)$,
then we have
$$Y_{{\hat{j}}+t+d_2-\frac{1}{2}}M^{\underline{x}}Y^{\underline{y}}v_{\underline{x},\underline{y},\underline{0}}=
M^{\underline{x}}[Y_{{\hat{j}}+t+d_2-\frac{1}{2}},Y^{\underline{y}}]v_{\underline{x},\underline{y},\underline{0}}.$$
Note that $M_tv_{\underline{x},\underline{y},\underline{0}}\neq0$.
If $\underline{y}=\underline{j}$, it is easy to get that
$$\mathrm{deg}(Y_{{\hat{j}}+t+d_2-\frac{1}{2}}M^{\underline{x}}Y^{\underline{y}}v_{\underline{x},\underline{y},\underline{0}})=
(\underline{x},\underline{y}^{\prime\prime},\underline{0})\preceq (\underline{i},\underline{j}^{\prime\prime},\underline{0}),$$
where the equality holds if and only if $\underline{x}=\underline{i}$.

Now suppose $(\underline{y},\mathbf{w}(\underline{y}))\prec(\underline{j},\mathbf{w}(\underline{j}))$, then we write  $$\mathrm{deg}(Y_{{\hat{j}}+t+d_2-\frac{1}{2}}M^{\underline{x}}Y^{\underline{y}}v_{\underline{x},\underline{y},\underline{0}})
=(\underline{x}_1,\underline{y}_1,\underline{0})\in \mathbb{M}\times \mathbb{M}\times \mathbb{M}.$$ If $\mathbf{w}(\underline{y})<\mathbf{w}(\underline{j})$,
then we get $\mathbf{w}(\underline{y}_1)\leq \mathbf{w}(\underline{y})-{\hat{j}}<\mathbf{w}(\underline{j})-{\hat{j}}=\mathbf{w}(\underline{j}^{\prime\prime})$,
which shows that $(\underline{x}_1,\underline{y}_1,\underline{0})\prec(\underline{i},\underline{j}^{\prime\prime},\underline{0})$.

Then we suppose $\mathbf{w}(\underline{y})=\mathbf{w}(\underline{j})$ and ${\underline{y}}\prec{\underline{j}}$. Let $\hat{y}:=\mathrm{min}\{s:y_s\neq0\}>0$.
If $\hat{y}>{\hat{j}}$, we obtain $\mathbf{w}(\underline{y}_1)<\mathbf{w}(\underline{y})-{\hat{j}}=\mathbf{w}(\underline{j}^{\prime\prime})$. If $\hat{y}={\hat{j}}$,
we can similarly check that $(\underline{x}_1,\underline{y}_1,\underline{0})=(\underline{x},\underline{y}^{\prime\prime},\underline{0})$.
By $\underline{y}^{\prime\prime}\prec \underline{j}^{\prime\prime}$, we have $\mathrm{deg}(Y_{{\hat{j}}+t+d_2-\frac{1}{2}}M^{\underline{x}}Y^{\underline{y}}v_{\underline{x},\underline{y},\underline{0}})
=(\underline{x}_1,\underline{y}_1,\underline{0})\prec(\underline{i},\underline{j}^{\prime\prime},\underline{0})$ in both cases.

Consequently, we conclude that
$\mathrm{deg}(Y_{{\hat{j}}+t+d_2-\frac{1}{2}}v)=(\underline{i},\underline{j}^{\prime\prime},\underline{0})$.

{\rm (3)} Noticing that $L_{\hat{i}+t+d_1}V=0$ and $[L_{\hat{i}+t+d_1},M_{-\hat{i}-d_1}]V\neq0$, then we have
 $$L_{\hat{i}+t+d_1}M^{\underline{x}}v_{\underline{x},\underline{0},\underline{0}}=
 a_{\hat{i}} M^{\underline{x}^\prime}[L_{\hat{i}+t+d_1},M_{-\hat{i}-d_1}]v_{\underline{x},\underline{0},\underline{0}}\quad \mbox{for} \ a_{\hat{i}}\in\C\backslash\{0\},  \underline{x}^\prime\in \mathbb{M}, (\underline{x},\underline{0},\underline{0})\in\mathrm{supp}(v),$$
which easily yields our result.  This proves Claim \ref{claim3.1}.

Using Claim \ref{claim3.1} repeatedly, from any nonzero element $v\in\mathrm{Ind}(V)$ we can reach a nonzero element in
 $\mathcal{U}(\mathcal{G})v\cap V\neq0$, which indicates the simplicity of $\mathrm{Ind}(V)$.
   Part  $(2)$ of Theorem \ref{th1} can be easily checked by a direct calculation.
\end{proof}

\begin{rema}\label{rema3.3} \rm
In Theorem \ref{th1}, when $t=0$, the condition (a) is equivalent to that $\nu_0\neq 0$. In addition, from the above proof, we see that Claim \ref{claim3.1} also holds without the assumption of the simplicity of $V$ as a $\G_{d_1,d_2}$-module.
\end{rema}

Moreover, we have the following corollary.
\begin{coro}
Letting $d_1,d_2$ and $V$ as in Theorem \ref{th1} except that $V$ may not be simple over $\G_{d_1,d_2}$, then we have
$$V=\{v\in\mathrm{Ind}(V)\mid M_iv=Y_{j-\frac{1}{2}}v=L_iv=0,\ \forall  i>t,j>t+d_2,k>t+d_1\}.$$
\end{coro}


Denote by $\mathcal{G}^{(x,y,z)}$ the subalgebra generated by $M_i,(1-\delta_{j,0})Y_{j-\frac{1}{2}},L_k$
   with  $i\geq x,j\geq y$ and $k\geq z$.
Now we are ready to state the second main result of this section.
\begin{theo}\label{th2}
Let $\nu_0\neq0$ and $S$ be a simple  $\G$-module. Then the following conditions are equivalent:

\noindent
{\rm (1)}  There exists $t\in\Z$ such that the actions of $M_i,(1-\delta_{i,0})Y_{i-\frac{1}{2}},L_i,i\geq t$ on $S$ are locally finite.

\noindent
{\rm (2)}    There exists $t\in\Z$ such that the actions of $M_i,(1-\delta_{i,0})Y_{i-\frac{1}{2}},L_i,i\geq t$ on $S$ are locally nilpotent.

\noindent
{\rm (3)}  There exist $x,y,z\in\Z$ such that  $S$ is a  locally finite $\mathcal{G}^{(x,y,z)}$-module.

\noindent
{\rm (4)}    There exist $x,y,z\in\Z$ such that  $S$ is a  locally nilpotent $\mathcal{G}^{(x,y,z)}$-module.

\noindent
{\rm (5)} There exist $d_1,d_2\in\N$ and  a simple $\G_{d_1,d_2}$-module $V$ satisfying the conditions in Theorem \ref{th1} such that
$S\cong\mathrm{Ind}(V)$.
\end{theo}
\begin{proof}
First we prove $(1)\Rightarrow(5)$. Suppose that $S$ is a simple $\G$-module and there exists $t\in\N$ such that the actions of
$M_i,(1-\delta_{i,0})Y_{i-\frac{1}{2}},$ and $L_i$ for all $i\geq t$ are locally finite.
Then we can choose a nonzero $v\in S$ such that $L_tv=\lambda v$ for some $\lambda\in \C$.

Take any $j\in\Z$ with $j>t$ and we  denote
 \begin{eqnarray*}
 &&N_M=\sum_{m\in\N}\C L_t^mM_jv=\mathcal{U}(\C L_t)M_jv,\\&&
  N_Y=\sum_{m\in\N}\C L_{t}^mY_{j-\frac{1}{2}}v=\mathcal{U}(\C L_{t})Y_{j-\frac{1}{2}}v,\\&&
   N_L=\sum_{m\in\N}\C L_t^mL_jv=\mathcal{U}(\C L_t)L_jv,
 \end{eqnarray*}
 which are all finite-dimensional. By the definition of \eqref{def1.1}, it is easy to   get
 \begin{eqnarray*}
 (j+mt)M_{j+(m+1)t}v&=&[L_t,M_{j+mt}]v\\&=&
 L_tM_{j+mt}v-M_{j+mt}L_tv=(L_t-\lambda)M_{j+mt}v,\\
 (j+mt-\frac{t+1}{2})Y_{j+(m+1)t-\frac{1}{2}}v&=&[L_t,Y_{j+mt-\frac{1}{2}}]v\\&=&
 L_tY_{j+mt-\frac{1}{2}}v-Y_{j+mt-\frac{1}{2}}L_tv=(L_t-\lambda)Y_{j+mt-\frac{1}{2}}v,\\
 (j+(m-1)t)L_{j+(m+1)t}v&=&[L_t,L_{j+mt}]v\\&
 =&L_tL_{j+mt}v-L_{j+mt}L_tv=(L_t-\lambda)L_{j+mt}v,\quad \forall m\in\N,
 \end{eqnarray*}
which imply that
\begin{eqnarray*}
&&M_{j+mt}v\in N_M\Rightarrow M_{j+(m+1)t}v\in N_M,\
Y_{j+mt-\frac{1}{2}}v\in N_Y\Rightarrow Y_{j+(m+1)t-\frac{1}{2}}v\in N_Y,\\&&
   L_{j+mt}v\in N_L \Rightarrow L_{j+(m+1)t}v\in N_L\quad \mbox{for\ all}\ m\in\N \ \mbox{and} \ j>t.
\end{eqnarray*}
 Therefore, by induction on $m$, we obtain
 $$M_{j+mt}v\in N_M,\  Y_{j+mt-\frac{1}{2}}v\in N_Y, \ L_{j+mt}v\in N_L, \quad \forall m\in \N.$$
Then, it follows from the facts that $\sum_{m\in\N}\C M_{j+mt}v$, $\sum_{m\in\N}\C Y_{j+mt-\frac{1}{2}}v$ and  $\sum_{m\in\N}\C L_{j+mt}v$ are
finite-dimensional for $j>t$. Hence,
 \begin{eqnarray*}
&&\sum_{i\in\N}\C M_{t+i}v=\C M_{t}v+\sum_{j=t+1}^{2t}\Big(\sum_{m\in\N}\C M_{j+mt}v\Big),\\&&
\sum_{i\in\N}\C Y_{t+i-\frac{1}{2}}v=\C Y_{t-\frac{1}{2}}v+\sum_{j=t+1}^{2t}\Big(\sum_{m\in\N}\C Y_{j+mt-\frac{1}{2}}v\Big), \\&&
\sum_{i\in\N}\C L_{t+i}v=\C L_{t}v+\sum_{j=t+1}^{2t}\Big(\sum_{m\in\N}\C L_{j+mt}v\Big)
\end{eqnarray*}
 are all finite-dimensional. In fact, we can take $l\in\Z_+$ such that
\begin{equation}\!\label{clym3.1}
\aligned
&\sum_{i\in\N}\C M_{t+i}v\!=\!\sum_{i=0}^{l}\C  M_{t+i}v,\
 \sum_{i\in\N}\C Y_{t+i-\frac{1}{2}}v\!=\!\sum_{i=0}^{l}\C  Y_{t+i-\frac{1}{2}}v,\
\sum_{i\in\N}\C L_{t+i}v\!=\!\sum_{i=0}^{l}\C  L_{t+i}v.
\endaligned\!\!
\end{equation}
Now we write $V^\prime=\sum_{x_0,\ldots,x_l,y_0,\ldots,y_l,z_0,\ldots,z_l\in\N}\C M_t^{x_0}\cdots
  M_{t+l}^{x_l}Y_{t-\frac{1}{2}}^{y_0}\cdots Y_{t+l-\frac{1}{2}}^{y_l}L_t^{z_0}\cdots L_{t+l}^{z_l}v$, which is finite-dimensional by (1).
 \setcounter{clai}{0}
 \begin{clai}\label{claim2}
 $V^\prime$ is a (finite-dimensional) $\G^{(t,t,t)}$-module.
 \end{clai}
 To prove the claim, using the $\mathrm{PBW}$ Theorem,  any   $M_{t+s}v^\prime,Y_{t+s-\frac{1}{2}}v^\prime,L_{t+s}v^\prime$  with $s\in\N$ and $v^\prime\in V^\prime$
can be written respectively as a sum of vectors of the form:
 \begin{equation}\label{clym3.2}
\aligned
&M_{t+s}M_t^{x_0}\cdots
  M_{t+l}^{x_l}Y_{t-\frac{1}{2}}^{y_0}\cdots Y_{t+l-\frac{1}{2}}^{y_l}L_t^{z_0}\cdots L_{t+l}^{z_l}v,\\&
  Y_{t+s-\frac{1}{2}}M_t^{x_0}\cdots
 M_{t+l}^{x_l}Y_{t-\frac{1}{2}}^{y_0}\cdots Y_{t+l-\frac{1}{2}}^{y_l}L_t^{z_0}\cdots L_{t+l}^{z_l}v,\\&
 L_{t+s} M_t^{x_0}\cdots
  M_{t+l}^{x_l}Y_{t-\frac{1}{2}}^{y_0}\cdots Y_{t+l-\frac{1}{2}}^{y_l}L_t^{z_0}\cdots L_{t+l}^{z_l}v.
\endaligned
\end{equation}
Now we prove that all elements above lie in $V^\prime$. By \eqref{clym3.1},
we only need to show that the elements in \eqref{clym3.2} with $0\leq s\leq l$ lie in $V^\prime$.
This is clear for the first element in \eqref{clym3.2}. For the second element  in \eqref{clym3.2}, we have
\begin{equation*}
\aligned
 &Y_{t+s-\frac{1}{2}}M_t^{x_0}\cdots
  M_{t+l}^{x_l}Y_{t-\frac{1}{2}}^{y_0}\cdots Y_{t+l-\frac{1}{2}}^{y_l}L_t^{z_0}\cdots L_{t+l}^{z_l}v\\
 =~&M_t^{x_0}\cdots
  M_{t+l}^{x_l}Y_{t-\frac{1}{2}}^{y_0}\cdots Y_{t+s-\frac{1}{2}}^{y_s+1}\cdots Y_{t+l-\frac{1}{2}}^{y_l}L_t^{z_0}\cdots L_{t+l}^{z_l}v\\&
+[Y_{t+s-\frac{1}{2}},M_t^{x_0}\cdots
  M_{t+l}^{x_l}Y_{t-\frac{1}{2}}^{y_0}\cdots Y_{t+s-\frac{1}{2}}^{y_{s}}]Y_{t+s+\frac{1}{2}}^{y_{s+1}}\cdots Y_{t+l-\frac{1}{2}}^{y_l}L_t^{z_0}\cdots L_{t+l}^{z_l}v.
\endaligned
\end{equation*}
Then, it is clear that the second element in \eqref{clym3.2}   lies in $V^\prime$.
 For the third element  in \eqref{clym3.2}, one can easily check that
\begin{equation*}
\aligned
 &L_{t+s}M_t^{x_0}\cdots
  M_{t+l}^{x_l}Y_{t-\frac{1}{2}}^{y_0}\cdots Y_{t+l-\frac{1}{2}}^{y_l}L_t^{z_0}\cdots L_{t+l}^{z_l}v\\
 =~&M_t^{x_0}\cdots
  M_{t+l}^{x_l}Y_{t-\frac{1}{2}}^{y_0}\cdots Y_{t+l-\frac{1}{2}}^{y_l}L_t^{z_0}\cdots L_{t+s}^{z_s+1}\cdots L_{t+l}^{z_l}v\\&
+[L_{t+s},M_t^{x_0}\cdots
  M_{t+l}^{x_l}Y_{t-\frac{1}{2}}^{y_0}\cdots Y_{t+l-\frac{1}{2}}^{y_{l}}L_t^{z_0}\cdots L_{t+s}^{z_{s}}]L_{t+s+1}^{z_{s}+1}\cdots L_{t+l}^{z_l}v.
\endaligned
\end{equation*}
Using  induction, we can show that  all  terms in above equation  lie in $V^\prime$.
Hence  we can get  the third element in \eqref{clym3.2} also  lies in $V^\prime$.
Then Claim \ref{claim2}  is obtained.

It follows from Claim 1 that we can  choose
 a minimal $n\in\N$ such that
 \begin{equation}\label{lm3.3}
 (L_m+a_1L_{m+1}+\cdots + a_{n}L_{m+{n}})V^\prime=0
 \end{equation}
 for some $m\geq t$ and  $a_i\in \C$.
 Applying $L_m$ to \eqref{lm3.3}, one has
$$(a_1[L_m,L_{m+1}]+\cdots +a_{n}[L_m,L_{m+n}])V^\prime=0,$$
which implies $n=0$, that is, $L_mV^\prime=0$.
Then we have
$$0=L_iL_{m}V^\prime=[L_i,L_{m}]V^\prime+L_mL_iV^\prime=(m-i)L_{m+i}V^\prime,\quad \forall i\geq t,$$
namely, $L_{m+i}V^\prime=0$ for all $i>m$. Similarly, we have $M_{m+i}V^\prime=0$
 for all $i>t$ and $Y_{m+i-\frac{1}{2}}V^\prime=0$  for all $1\neq i>m$, respectively.
 For any $\tilde{i},\tilde{j},\tilde{k}\in \Z$, we consider the vector space
 $$N_{\tilde{i},\tilde{j},\tilde{k}}=\{v\in S\mid M_iv=(1-\delta_{j,0})Y_{j-\frac{1}{2}}v=L_kv=0 \quad \mbox{for\ all}\ i>\tilde{i},j>\tilde{j},k>\tilde{k}\}.$$
Clearly, $N_{{\tilde{i},\tilde{j},\tilde{k}}}\neq0$ for sufficiently large ${\tilde{i},\tilde{j}},\tilde{k}\in\Z$.
On the other hand, $N_{\tilde{i},\tilde{j},\tilde{k}}=0$ for all $\tilde{i}<0$ since we have $M_0v=\nu_0 v\neq0$ for any nonzero $v\in S$. Thus we can find a smallest nonnegative integer, saying $r_1$, and choose some  $r_2,r_3\geq r_1$ with $r_3-r_1\geq2(r_2-r_1)-1$ such that  $N_{r_1,r_2,r_3}\neq 0$. Denote  $d_1=r_3-r_1, d_2=r_2-r_1$ and  $V=N_{r_1,r_2,r_3}$.
Using $k>r_3,j>r_2$ and $l\geq 1$, it follows from $k+l-d_2-\frac{1}{2}>r_3+\frac{1}{2}-d_2\geq r_2-\frac{1}{2}$ and $j+l-d_2-1>r_2-d_2=r_1$  that we can easily check that
$$L_{k}(Y_{l-d_2-\frac{1}{2}}v)=(l-d_2-\frac{k+1}{2})Y_{k+l-d_2-\frac{1}{2}}v=0$$
and
$$
Y_{j-\frac{1}{2}}(Y_{l-d_2-\frac{1}{2}}v)=(l-d_2-j)M_{j+l-d_2-1}v=0,
$$
respectively.
Clearly, $Y_{l-d_2-\frac{1}{2}}v\in V$ for
all $l\geq 1$. Similarly, we can also obtain $M_{e-d_1}v\in V$ and $L_{e}v\in V$
for all $e\in \N$. Therefore, $V$ is a $\G_{d_1,d_2}$-module.

By the definition of $V$, we can obtain that the action of  $M_{t}$  on $V$ is injective. Since $S$ is simple
and generated by $V$, then there exists a canonical surjective map
$$\pi:\mathrm{ Ind}(V) \rightarrow S, \quad \pi(1\otimes v)=v,\quad \forall  v\in V.$$
Next we only need to show that $\pi$ is also injective, that is to say, $\pi$ as the canonical map is bijective.  Let $K=\mathrm{ker}(\pi)$. Obviously, $K\cap V=0$. If	$K\neq0$, we can choose a nonzero vector $v\in K\setminus V$ such that $\mathrm{deg}(v)=(\underline{i},\underline{j},\underline{k})$ is minimal possible.
Note that $K$ is a $\mathcal{G}$-submodule of $\mathrm{Ind}(V)$.
By Claim \ref{claim3.1} in  Theorem \ref{th1} and Remark \ref{rema3.3} we can create a new vector $u\in K$  with $\mathrm{deg}(u)\prec(\underline{i},\underline{j},\underline{k})$, which is a contradiction. This forces $K=0$,
that is, $S\cong \mathrm{Ind}(V)$. According to  the property of induced modules, we see that $V$ is simple as a $\mathcal{G}_{d_1,d_2}$-module.

Moreover, $(5)\Rightarrow(3)\Rightarrow(1)$, $(5)\Rightarrow(4)\Rightarrow(2)$
 and $(2)\Rightarrow(1)$ are clear. This completes the proof of the theorem.
\end{proof}

\begin{remark} \rm
From the above proof, we know that any simple module satisfying conditions in
Theorem \ref{th2} is determined by some simple module $V$ over a certain subalgebra $\G_{d_1,d_2}$.
The conditions of Theorem  \ref{th1}  imply that  $V$  can be viewed as
a simple module over some finite-dimensional solvable quotient algebra of $\G_{d_1,d_2}$. This reduces the study of such modules over $\mathcal{G}$ to the study of simple modules over the corresponding finite-dimensional algebras.
\end{remark}

\section{Some examples}
In this section, some examples of simple $\G_{d_1,d_2}$-modules are given. Then, by Theorem \ref{th1}, we can construct many new simple $\mathcal{G}$-modules.

First we describe highest weight modules and Whittaker modules as follows.
\begin{exam}
 Let ${\mathfrak{h}}=\mathrm{span}_{\C}\{L_0,M_0,C\}$ be the Cartan subalgebra of $\G$. For $\xi=(\xi(L_0),\nu_0\neq0,c)\in{\mathfrak{h}}^*$, we have the
Verma module $\mathcal{M}(\xi)=\mathcal{U}(\G)\otimes_{\mathcal{U}({\mathfrak{h}}+\G_{0,0})}\C_\xi$, where $L_i\C_\xi=M_i\C_\xi=Y_{i-\frac{1}{2}}\C_\xi$  for $i > 0$,
while $L_0\C_\xi=\xi(L_0)\C_\xi,M_0\C_\xi=\nu_0\C_\xi$ and $C\C_\xi=c\C_\xi$, respectively.
 The module $\mathcal{M}(\xi)$ has the
unique simple quotient  $\mathcal{L}(\xi)$, the unique (up to isomorphism) simple highest weight
module with highest weight $\xi$.  
These
modules correspond to the case $t=0$ in Theorem \ref{th2}.
\end{exam}

\begin{exam}
 Consider a nonzero $\xi:=(\lambda_1,\lambda_2,\mu_1,\mu_2,\nu_0\neq0,\nu_1\neq0,c)\in\C^7$.  Denote by $\mathcal{V}_\xi$
 the $\G_{1,1}$-module $\mathcal{U}(\G_{1,1})/I$,
where $I$  is the left ideal generated by $L_1-\lambda_1,L_2-\lambda_2,L_3,\ldots,
Y_{\frac{1}{2}}-\mu_1,Y_{\frac{3}{2}}-\mu_2,Y_{\frac{5}{2}},\ldots,M_1-\nu_1,M_2,\ldots,M_0-\nu_0,C-c$.
 It is easy to see that $\mathcal{V}_\xi$ is simple.
 The module  $\mathcal{V}_\xi$ obviously satisfies the conditions of Theorem \ref{th1} (with $t=d_1=d_2=1$).
 Hence, it follows from Theorem \ref{th1} that we obtain the corresponding simple induced $\G$-module $\mathrm{Ind}(\mathcal{V}_\xi)$.
When $c=0$, these are exactly the Whittaker modules over $\G$ constructed in  \cite{ZTL}.
\end{exam}

Now we consider some $t\in \Z_+,d_1,d_2\in\N$, and choose subsets $S_\lambda \subseteq \{1,\ldots,t+d_1\},S_\mu\subseteq \{-d_2+1,\ldots,t+d_2\}$ and disjoint subsets
$S_{\nu,0},S_{\nu,1}\subseteq \{-d_1,\ldots,t\}$ with $0,t\in S_{\nu,1}$. Set $S_\nu=S_{\nu,0}\cup S_{\nu,1}$.
Let $\bar{S}_\lambda=\{0,1,\ldots,t+d_1\}\setminus S_\lambda, \bar{S}_\mu=\{-d_2,\dots,t+d_2\}\setminus S_\mu$ and $\bar{S}_\nu=\{-d_1,\ldots,t\}\setminus S_\nu$.
Set $c\in \C$, $\lambda_i\in\C,i\in S_\lambda,\mu_i\in\C,i\in S_\mu$ and $\nu_i\in\C,i\in S_\nu$. Moreover, $i\in S_\nu$ with $\nu_i\neq0$ if and only if $i\in S_{\nu,1}$.
In addition, the following conditions are satisfied:

\noindent
{\rm (I)} for all $i,j\in S_\lambda,i\neq j$, we either have $i+j> t+d_1$  or  $i+ j\in S_\lambda$ and $\lambda_{i+j}=0$;

\noindent
{\rm (II)} for all $i\in S_\lambda,j\in S_\mu,\frac{i}{2}\neq j-\frac{1}{2}$, we either have  $i+j>t+d_2$ or $i+ j\in S_\mu$ and $\mu_{i+j}=0$;

\noindent
{\rm (III)} for all $i,j\in S_\mu,i\neq j$, we either have   $i+j-1> t$ or $i+ j-1\in S_{\nu,0}$ and $\nu_{i+j-1}=0$;

\noindent
{\rm (IV)} for all $i\in S_\lambda,j\in S_\nu\setminus \{0\}$,  we either have $i+j> t$ or  $i+ j\in S_{\nu,0}$ and $\nu_{i+j}=0$;

\noindent
{\rm (V)} for any $j\in \bar{S}_\lambda$, there exists a nonzero $i\in S_\nu$ such that $i+j\in \bar{S}_\nu\cup S_{\nu,1}$
 and  $i+j^\prime\in S_{\nu,0}$
 for all $j^\prime\in\bar{S}_\lambda$ with $j<j^\prime<t-i$ (if $t-i\in\bar{S}_\lambda$ we replace  $t-i$ by $j$ and these $j$ such that $i+j\in S_{\nu,1}$);

\noindent
{\rm (VI)} for any $j\in \bar{S}_\mu$, there exists $i\in S_\mu$ such that $i+j-1\in \bar{S}_\nu\cup S_{\nu,1}$
 and $i+j^\prime-1\in S_{\nu,0}$ for all $j^\prime\in\bar{S}_\mu$ with
 $j<j^\prime< t-i+1$  (if $t-i+1\in\bar S_\mu$ we  replace  $t-i+1$ by $j$  and these $j$ such that $i+j-1\in S_{\nu,1}$);


\noindent
{\rm (VII)} for any $j\in \bar{S}_\nu$, there exists $i\in S_\lambda$ such that $i+j\in \bar{S}_\nu\cup S_{\nu,1}$
 and    $i+j^\prime\in S_{\nu,0}$ for all $j^\prime\in \bar{S}_{\nu}$ with
 $j<j^\prime< t-i$ (if $t-i\in\bar S_\nu$ we replace  $t-i$  by $j$  and these $j$ such that $i+j\in S_{\nu,1}$).

For any $\underline{i}=(i_1,\ldots,i_t),\underline{j}=(j_1,\ldots,j_t),\underline{k}=(k_1,\ldots,k_t)\in \N^t,t\in\Z_+$, we can define the lexicographical order (which is {\em not} the one defined in Section 2) on $\N^t$ as follows:
\begin{equation*}
\underline{i}>\underline{j} \ \Leftrightarrow \  \mathrm{there}\ \mathrm{exists}\ r \ \mathrm{such}\  \mathrm{that} \
(i_s=j_s,\ \forall  1\leq s< r)\  \mathrm{and} \ i_r>j_r.
\end{equation*}
Denote $\epsilon_r\in\N^t$ with $1$ in the $r$'th position and $0$ elsewhere.

For any set $X=\{x_1,\ldots,x_n\}$, we denote by $|X|$ the number of  elements in $X$. Now we set $m=|\bar{S}_\lambda|,n=|\bar{S}_\mu|,l=|\bar{S}_\nu|$. Let $\bar{S}_\lambda=\{p_1=0,p_2,\ldots,p_m\}$
with $p_1<\cdots<p_m$,  $\bar{S}_\mu=\{q_1=-d_2,q_2,\ldots,q_n\}$ with $q_1<\cdots<q_n$  and
 $\bar{S}_\nu=\{r_1,r_2,\ldots,r_l\}$ with $r_1<\cdots<r_l$. Denote by $Q$ the $\G_{d_1,d_2}$-module $\mathcal{U}(\G_{d_1,d_2})/I$,
 where $I$ is the left ideal generated by $L_i-\lambda_i,Y_{j-\frac{1}{2}}-\mu_j,M_k-\nu_k,C-c$ with
 $i\in\N\setminus \bar{S}_\lambda,j\in (\N-d_2)\setminus \bar{S}_\mu,k\in (\N-d_1) \setminus \bar{S}_\nu$,
 where we make the convention that $\lambda_i=0$ for $i\notin S_\lambda$, $\mu_j=0$ for $j\notin S_\mu$
 and  $\nu_k=0$ for $k\notin S_{\nu,1}$.
  By the $\mathrm{PBW}$ Theorem,  a basis of $Q$ is given by the images of
 $$L^{\underline{i}}Y^{\underline{j}}M^{\underline{k}}=L^{i_1}_{p_1}\cdots L^{i_m}_{p_m}Y^{j_1}_{q_1-\frac{1}{2}}\cdots Y^{j_n}_{q_n-\frac{1}{2}}M^{k_1}_{r_1}\cdots M^{k_l}_{r_l},$$
 where  $\underline{i}:=(i_1,\cdots,i_m)\in\N^m,\underline{j}:=(j_1,\cdots,j_n)\in\N^n,\underline{k}:=(k_1,\cdots,k_l)\in\N^l$.
  Then a typical element $v\in Q$ can be written as
 \begin{equation}\label{v4.1}
 v=\sum_{\underline{i}\in\N^m,\underline{j}\in\N^n,\underline{k}\in\N^l}a_{\underline{i},\underline{j},
 \underline{k}}L^{\underline{i}}Y^{\underline{j}}M^{\underline{k}},
 \end{equation}
 with only finitely many $a_{\underline{i},\underline{j},
 \underline{k}}$ nonzero. Set $\mathrm{supp}(v)=\{(\underline{i},\underline{j},\underline{k})\mid a_{\underline{i},\underline{j},
 \underline{k}}\neq0\}$ and denote by $\mathrm{deg}(v)$ the maximal element in $\mathrm{supp}(v)$ under the following total order
\begin{equation*}
\aligned
(\underline{i},\underline{j},\underline{k})\succ(\underline{l},\underline{m},\underline{n})\ \Leftrightarrow \
  \underline{i}>\underline{l}\quad
 \mathrm{or}\quad \underline{i}=\underline{l} \ \mathrm{and} \ \underline{j}> \underline{m}
  \quad \mathrm{or}\quad \underline{i}=\underline{l}, \underline{j}= \underline{m}\ \mathrm{and}\ \underline{k}> \underline{n}
\endaligned
\end{equation*}
for any $(\underline{i},\underline{j},\underline{k}),(\underline{l},\underline{m},\underline{n})\in \N^m \times\N^n \times\N^l.$
We also make the convention that the zero element does not have a degree. Now  we can prove the following lemma.
 \begin{lemm}\label{lemm4.1}
The $\G_{d_1,d_2}$-module $Q$ is simple.
\end{lemm}
\begin{proof}
 Since $\bar{S}_\lambda$ is not empty, $Q\neq0$. By conditions ${\rm (I)}$-${\rm (IV)}$, $Q$ is a module of $\G_{d_1,d_2}$.  Let $v\in Q\setminus \{0\}$ and we can write $v$   as in \eqref{v4.1} that is a nonzero linear combination of basis elements. Set $\mathrm{deg}(v)=(\underline{i},\underline{j},\underline{k})$.
\begin{case}
 $\underline{i}\neq\underline{0}$.
Setting $x:=\mathrm{min}\{s>0\mid i_s\neq0\}$, then we have  $p_x\in \bar{S}_\lambda$.
According to condition ${\rm (V)}$, there exists nonzero $y\in S_\nu$ such that
 $p_x+y\in\bar{S}_\nu\cup S_{\nu,1}$ and $y+j^\prime\in S_{\nu,0}$ for all  $j^\prime\in\bar{S}_\lambda$ with
 $p_x<j^\prime< t-y$  (if $t-y\in\bar{S}_\lambda$ we replace  $t-y$ by $p_x$ and these $p_x$ such that $p_x+y\in S_{\nu,1}$). Applying $M_y-\nu_y$ to $v$, which gives
 $$
\mathrm{deg}((M_y-\nu_y)v)=\left\{\begin{array}{llll}
(\underline{i}-\epsilon_x,\underline{j},\underline{k}+\epsilon_s),&\mbox{if}\
p_x+y=r_s\in \bar{S}_\nu,\\[4pt]
(\underline{i}-\epsilon_x,\underline{j},\underline{k}),&\mbox{if}\
p_x+y\in {S}_{\nu,1}.
\end{array}\right.
$$
\end{case}

\begin{case}
 $\underline{i}=\underline{0}$.
Setting $x:=\mathrm{min}\{s>0\mid j_s\neq0\}$, then  we have $q_x\in \bar{S}_\mu$.
According to  condition ${\rm (VI)}$, there exists $y\in S_\mu$ such that
 $q_x+y-1\in\bar{S}_\nu\cup S_{\nu,1}$ and $y+j^\prime-1\in S_{\nu,0}$ for all $j^\prime\in\bar{S}_\mu$ with
$q_x<j^\prime< t-y+1$ (if $t-y+1\in\bar S_\mu$ we replace $t-y+1$ by $q_x$ and these $q_x$ such that $q_x+y-1\in S_{\nu,1}$). Applying $Y_{y-\frac{1}{2}}-\mu_{y}$ to $v$, which gives
 $$
\mathrm{deg}((Y_{y-\frac{1}{2}}-\mu_{y})v)=\left\{\begin{array}{llll}
(\underline{0},\underline{j}-\epsilon_x,\underline{k}+\epsilon_s),&\mbox{if}\
q_x+y=r_s\in \bar{S}_\nu,\\[4pt]
(\underline{0},\underline{j}-\epsilon_x,\underline{k}),&\mbox{if}\
q_x+y\in {S}_{\nu,1}.
\end{array}\right.
$$
\end{case}

\begin{case}
 $\underline{i}=\underline{j}=\underline{0}$.
Setting $x:=\mathrm{min}\{s>0\mid k_s\neq0\}$, then we have $r_x\in \bar{S}_\nu$.
According to  condition ${\rm (VII)}$, there exists $y\in S_\lambda$ such that
 $r_x+y\in\bar{S}_\nu\cup S_{\nu,1}$ and $y+j^\prime$ for all $j^\prime\in\bar{S}_\nu$  with $r_x<j^\prime< t-y$ (if $t-y\in\bar S_\nu$ we replace $t-y$ by $r_x$ and these $r_x$ such that $r_x+y\in S_{\nu,1}$). Applying $L_y-\lambda_y$ to $v$, which gives
 $$
\mathrm{deg}((L_y-\lambda_y)v)=\left\{\begin{array}{llll}
(\underline{0},\underline{0},\underline{k}-\epsilon_x+\epsilon_s),&\mbox{if}\
r_x+y=r_s\in \bar{S}_\nu,\\[4pt]
(\underline{0},\underline{0},\underline{k}-\epsilon_x),&\mbox{if}\
r_x+y\in {S}_{\nu,1}.
\end{array}\right.
$$
\end{case}
Repeating this
process inductively (with respect to the degree of $v$), we can reach a nonzero element with degree $\underline{0}$ and this element can generate the whole module $Q$. Therefore $Q$ is a simple $\G_{d_1,d_2}$-module.
 \end{proof}
 \begin{remark}\label{re4.2}
Note that the module $Q$ obviously satisfies the conditions of Theorem \ref{th1}. Then one can  obtain the corresponding simple induced $\mathcal{G}$-module $\mathrm{Ind}(Q)$. It follows from Theorem \ref{th1} that  the actions of all the elements $M_i,Y_{j-\frac{1}{2}},L_k$ for $i>t,j>t+d_2,k>t+d_1$
are locally nilpotent and the actions of all the elements $M_i,Y_{j-\frac{1}{2}},L_k$ for  $i<-d_1,j<-d_2,k<0$ are injective and free on $\mathrm{Ind}(Q)$ (hence non-locally finite). Moreover, by using Lemma \ref{lemm4.1},  the actions of all
the elements $M_i,Y_{j-\frac{1}{2}},L_k$ for $i\in S_{\nu,0},j\in S_\mu,k\in S_\lambda,$ where $\nu_i=\mu_j=\lambda_k=0$ are locally nilpotent, and the actions of all the elements $M_i,Y_{j-\frac{1}{2}},L_k$ for $i\in S_{\nu,1},j\in S_\mu,k\in S_\lambda,$ where $\nu_i,\mu_j,\lambda_k\neq0$ are locally finite. By these facts, we  construct a lot of new simple modules, which are not isomorphic to the
simple weight modules and simple Whittaker modules (as far as we know simple $\G$-modules).
 \end{remark}
When $d_1,d_2>0$ or $t>1$, the simple modules $\mathrm{Ind}(Q)$ are new simple $\mathcal{G}$-modules as we stated in Remark \ref{re4.2}.
Finally, we give an example about new simple modules.
 \begin{exam}
For $d_1=d_2=0$ and $t=2$,  we can define
\begin{eqnarray*}
S_\lambda=\{2\},\ \bar{S}_\lambda=\{0,1\}, \ S_\mu=\{1\},\ \bar{S}_\mu=\{0,2\}, \ S_{\nu,1}=\{0,2\},\ S_{\nu,0}=\{1\},\ \bar{S}_\nu=\emptyset
\end{eqnarray*}
and $\lambda_2, \mu_1,  \nu_0, \nu_1, \nu_2, c\in\C$ such that  $\nu_1=0$, $\nu_0,\nu_2\neq0$. Then, we obtain a $\mathcal{G}_{0,0}$-module
$Q$. It is easy to see that all conditions ${\rm (I)}$-${\rm (VII)}$ are satisfied. Then $\mathcal{G}_{0,0}$-module
$Q$ is a simple module. Therefore, the induced module $\mathrm{Ind}(Q)$ is simple, which is a new simple module of $\G$.
\end{exam}

\section{The $W$-algebra $W(2,2)$}
In this section, we shall construct a large class of simple modules over the $W$-algebra $W(2,2)$.  As a result, we not only recover many known simple modules including highest weight modules and Whittaker modules that presented in \cite{WL} and \cite{ZD},  but also construct a lot
of new simple modules for  the $W$-algebra $W(2,2)$.

 The {\it $W$-algebra  $W(2,2)$}
$\mathcal{W}$ is defined to be a Lie algebra
 with a  $\C$-basis $\{L_{m},W_{m},C_{W}\mid m\in
\Z\}$ and the following nonvanishing Lie brackets:
\begin{eqnarray*}
&&[L_{m},L_{n}] = (n-m) L_{m+n}+ \delta_{m+n,0} \frac{1}{12}
(m^3 - m)C_{W},\\&&
[L_{m},W_{n}] = (n-m) W_{m+n}+ \delta_{m+n,0} \frac{1}{12}
(m^3 - m)C_{W}.
\end{eqnarray*}
It was introduced by Zhang and Dong in \cite{ZD} for the study of the classification of
moonshine type vertex operator algebras generated by two weight $2$ vectors.

Using lexicographical total order and  reverse lexicographical total order defined in Section 2,
 we induce a   {\em principal total order}  (which is {\em not} the one we described in Section 2) on $\mathbb{M}\times\mathbb{M}$, and denoted by $\succ$:
\begin{equation*}\aligned
(\underline{i},\underline{j}) \succ (\underline{l},\underline{m})\  \Leftrightarrow \ \   &
 (\underline{j},\mathbf{w}(\underline{j})) \succ (\underline{m},\mathbf{w}(\underline{m}))
\quad \mathrm{or}\quad \underline{j}=\underline{m} \
\mathrm{and }\ \underline{i} > \underline{l},\quad \forall \underline{i},\underline{j},\underline{l},\underline{m}\in\mathbb{M}.
\endaligned
\end{equation*}

We shall construct some new simple $\mathcal{W}$-modules with locally finite actions of elements $W_i,L_j$ for sufficiently large $i,j\in \Z$.
For any $d\in\N$, we write
$$\mathcal{W}_d=\sum_{i\in\N}(\C W_{i-d}\oplus\C L_i)\oplus\C C_W.$$
Letting $V$ be a simple $\mathcal{W}_d$-module, then we have the  induced   $\mathcal{W}$-module
$$\mathrm{Ind}(V)=\mathcal{U}(\mathcal{W})\otimes_{\mathcal{U}(\mathcal{W}_d)}V.$$
Moreover, we always assume that the actions of $W_0$ and $C_W$ on $V$ are scalars $h_W$ and $c_W$ respectively.
Now we  give a description of the following  results for the $W$-algebra $W(2,2)$.
\begin{theo}\label{th51}
Assume that $d\in\N$ and $V$ is a simple $\mathcal{W}_d$-module. If  there exists  $t\in\N$  such that

\noindent
{\rm (a)}  $
\left\{\begin{array}{llll}  \mbox{the\ action\ of}\ W_t\ \mbox{on}\ V \mbox{is\  injective},&\mbox{if}\ t\neq0,\\[4pt]
2h_W+\frac{n^2-1}{12}c_W\neq0,\ \forall n\in\Z\setminus\{0\},&\mbox{if}\ t=0,
\end{array}\right.
$

\noindent
{\rm (b)}   $W_iv=L_jv=0$ for  all $i>t$  and  $j>t+d$,

\noindent
then we have

\noindent
{\rm (1)} $\mathrm{Ind}(V)$ is a simple $\mathcal{W}$-module;

\noindent
{\rm (2)}  the actions of $W_i,L_j$  on $\mathrm{Ind}(V)$  for all  $i>t$  and  $j>t+d$  are locally nilpotent.
\end{theo}


Denote by $\mathcal{W}^{(x,y)}$ the subalgebra generated by $W_i,L_j$
   with  $i\geq x$ and $j\geq y$.
Subsequently, we characterize  simple modules over the $W$-algebra $W(2,2)$.
\begin{theo}\label{th52}
Suppose that $S$ is a simple  $\mathcal{W}$-module with  $2h_W+\frac{n^2-1}{12}c_W\neq0$ for all  $n\in\Z\setminus\{0\}$. Then the following conditions are equivalent:

\noindent
{\rm (1)}  There exists $t\in\Z$ such that the actions of $W_i,L_i,i\geq t$ on $S$ are locally finite.

\noindent
{\rm (2)}    There exists $t\in\Z$ such that the actions of $W_i,L_i,i\geq t$ on $S$ are locally nilpotent.

\noindent
{\rm (3)}  There exist $x,y\in\Z$ such that  $S$ is a  locally finite $\mathcal{W}^{(x,y)}$-module.

\noindent
{\rm (4)}    There exist $x,y\in\Z$ such that  $S$ is a  locally nilpotent $\mathcal{W}^{(x,y)}$-module.

\noindent
{\rm (5)} There exists $d\in\N$  and a simple $\mathcal{W}_d$-module $V$ satisfying the conditions in Theorem \ref{th51} such that
$S\cong\mathrm{Ind}(V)$.
\end{theo}

\end{document}